\documentclass[12pt]{amsart}
\usepackage{amssymb}
\usepackage{amsfonts}
\usepackage{latexsym}
\usepackage{amsmath}

\usepackage{color}

\usepackage{graphicx}

\newtheorem{thm}{Theorem}[section]

\newtheorem{corollary}{Corollary}
\newtheorem{lemma}[thm]{Lemma}

\newtheorem{theorem}[thm]{Theorem}
\newtheorem{proposition}{Proposition}
\newtheorem{claim}{Claim}

\newtheorem{question}{Question}

\newtheorem{remark}{Remark}

\newtheorem{example}{Example}

\numberwithin{equation}{section}

\newcommand{\bbC}{{\Bbb C}}
\newcommand{\CC}{{\widehat{\bbC}}}
\newcommand{\C}{{\mathbb C}}

\newcommand{\R}{{\mathbb R}}

\newcommand{\N}{{\mathbb N}}

\begin{document}

\title[Backward Iteration Algorithms for M\"obius Julia sets]{Backward Iteration Algorithms for Julia sets of M\"obius Semigroups}

\author{Rich Stankewitz}
\address[Rich Stankewitz]
    {Department of Mathematical Sciences\\
    Ball State University\\
    http://rstankewitz.iweb.bsu.edu/}
\email{rstankewitz@bsu.edu}

\author{Hiroki Sumi}
\address[Hiroki Sumi]
    {Department of Mathematics\\
    Graduate School of Science\\
    Osaka University\\
    1-1, Machikaneyama, Toyonaka\\
    Osaka, 560-0043, Japan\\
    http://www.math.sci.osaka-u.ac.jp/$\sim $sumi/welcomeou-e.html}
\email{sumi@math.sci.osaka-u.ac.jp}

\begin{abstract}
We extend a result regarding the Random Backward Iteration algorithm for drawing Julia sets (known to work for certain rational semigroups containing a non-M\"obius element) to a class of M\"obius semigroups which includes certain settings not yet been dealt with in the literature, namely, when the Julia set is not a thick attractor in the sense given in~\cite{FMS}.
\end{abstract}

\date{ September 1, 2016. Published in Discrete and Continuous Dynamical Systems Ser. A, 
{\bf 36}, 2016, 6475 - 6485.      }

\maketitle

\section{Introduction}
In this paper we consider the dynamics of rational semigroups (i.e.,
semigroups of rational maps on the Riemann sphere $\CC $ where the semigroup operation is
the composition of maps). In particular, we consider the algorithm to draw the figures of the Julia sets of
finitely generated  M\"obius semigroups (i.e., semigroups of M\"obius maps on $\CC $).
In~\cite{RandomBackStankewitzSumi} two methods for generating graphical computer approximations of Julia sets of finitely generated rational semigroups were discussed, the \textit{full backward iteration algorithm} and the \textit{random backward iteration algorithm}.  The former method was justified by the work of Boyd~\cite{Boyd2} and later generalized by Sumi~\cite{Su3}, and the latter method was justified by the present authors in~\cite{RandomBackStankewitzSumi}.  (See also~\cite{BarnsleyFractalsEvery, BarnsleyDemko} where these methods were first explored.
For the related topics in smooth ergodic theory, see
\cite{Mihailescu}.)  However, because of complexities that do not exist when there is a map of degree two or more, neither method was verified for classes of M\"obius semigroups in these papers.  The goal of the present paper, however, is to verify both methods for certain classes of M\"obius semigroups (see Theorems~\ref{FullMethod} and~\ref{main} for precise statements).

The \textit{full backward iteration algorithm} creates successive approximations to the Julia set as follows.  Starting with a seed value $z_0 \in \C$, the set $A_1$ of all preimages of $z_0$ under all generating maps $\{f_j:j=1, \dots, k\}$ of the semigroup is constructed.  (Note, our use of the word \emph{preimage} refers to a preimage of order 1 of $z_0$, i.e., a
point $y$ such that $f_j(y) = z_0$ for some $j$, as opposed to preimages of order $n \ge 2$, i.e., $y$ such that $f_{j_1} \circ \dots \circ f_{j_n}(y) = z_0$.)  Iteratively, the set $A_n$ is constructed to be the set of all preimages (of order 1) of all points in $A_{n-1}$ under all generating maps.  Hence, $A_n$ is the set of all preimages of order $n$ of $z_0$.  In general, for large $n$ the set $A_n$ approximates the Julia set of the semigroup (see precise statements in Section 2 where the $A_n$'s are exactly the supports of a corresponding convergent sequence of measures).

The \textit{random backward iteration algorithm}  (also known as the ``ergodic method" or ``chaos game" method)  creates successive approximations to the Julia set as follows.  Starting with a seed value $z_0 \in \C$ a random walk $\{z_n\}$ is generated by setting $z_n$ to be the outcome of randomly selecting one preimage of $z_{n-1}$ under a randomly selected generating map.  In general, the plotted points of this random walk then will give a visual approximation to the Julia set of the semigroup (see precise statements in Section 3 given in terms of convergence of certain measures).

Analogous results, for both the ``full" and ``random" methods, in the context of attractor sets for contracting iterated function systems (see~\cite{Hutchinson, BEH, Elton, RSThesis}) and Julia sets of  iterated single rational functions (see~\cite{HawkinsTaylor, Lyu, Mane, FLM}) are known.  Both the full and random methods for both attractor sets of iterated function systems and Julia sets of rational semigroups are implemented using the freely available application Julia 2.0~\cite{Julia2.0}.

We now introduce the basic terminology and notions needed to describe the full and random methods mentioned above.

For the entire paper we let $G=\langle f_1, \dots, f_k \rangle$ be a rational semigroup generated by M\"obius maps $f_j$ where the semigroup operation is
the composition of maps, i.e., $G$ is the collections of all maps which can be expressed as a finite composition of maps from the generating set $\{f_j:j=1, \dots, k\}$.  Also, since the results depend not on just the semigroup $G$ but also on the particular choice of generators, we assume each such semigroup comes with a particular fixed generating set of maps.

Research on the dynamics of rational
semigroups was initiated by Hinkkanen and Martin
in~\cite{HM1} (see also~\cite{HMJ}), and this remains a primary source for background information.  However, they studied only rational semigroups containing at least one element of degree at least two.  For an in depth look at the dynamics of M\"obius semigroups see~\cite{FMS}.

We follow~\cite{HM1} in saying that the \textit{Fatou set} $F(G)$ is the set of points in $\CC$ which have a neighborhood on which $G$ is normal, and its complement in $\CC$ is called the \textit{Julia set} $J(G)$.

We quote the following results from~\cite{HM1}. The Fatou set $F(G)$
is \textit{forward invariant} under each element of $G$, i.e.,
$g(F(G)) \subseteq F(G)$ for all $g \in G$, and thus $J(G)$ is
\textit{backward invariant} under each element of $G$, i.e.,
$g^{-1}(J(G)) \subseteq J(G)$ for all $g \in G$.
We note that the sets $F(G)$ and $J(G)$ are,
however, not necessarily completely invariant under the elements of
$G$.

We define the \emph{kernel Julia set} of $G$ by
$J_{ker}(G) =\cap_{g \in G} g^{-1} (J(G))$.

For a subset $A$ of $\CC $, we set $G(A)=\bigcup _{g\in G}g(A).$
Also, we set $G^{-1}:=\{ g^{-1}: g\in G\} $ and this is called the
inverse semigroup of $G$. Note that $G^{-1}$ is  a M\"obius semigroup
generated by $\{ f_{1}^{-1},\ldots,f_{k}^{-1}\} .$

The \emph{exceptional set} $E(G)$ is defined to be the set of points $z$ with a finite backward orbit $G^{-1}(z)=\cup_{g \in G} g^{-1}(\{z\})$.

We record the following well-known facts for later use.

\begin{proposition}\label{facts}
\textrm{ }
\begin{enumerate}
\item $J_{ker}(G)$ is the largest forward invariant subset
of $J(G)$ under the action of $G$.

\item For $z \in \CC \setminus E(G)$ we have $\overline{G^{-1}(z)} \supseteq J(G)$.  In particular, for $z \in J(G) \setminus E(G)$ we have $\overline{G^{-1}(z)} = J(G)$. \label{BackOrbitNonExcept}

\item $G(E(G)) = E(G)$. \label{ExcSetCompInv}

\item If $\#J(G) \geq 3$, then $J(G)$ is both perfect (thus uncountable) and the closure of the set of points that are repelling fixed points under any map in $G$. \label{RepelDense}

\end{enumerate}

\end{proposition}

\begin{proof}
(1) follows from the definition of $J_{ker}(G)$ and the fact that $J(G)$ is backward invariant under $G$.  (2) is shown in~\cite{HM1}.  Because $E(G)$ is clearly backward invariant, (3) follows from Proposition 2.17 in~\cite{FMS} when $E(G)$ is finite.  Now suppose $E(G)$ is infinite.  Since any three points in $E(G)$ must have finite backward orbits under each $g \in G$, we must then have each $g\neq Id$ in $G$ is elliptic of finite order.  This implies then that $g^{-1} \in G$ for each $g \in G$ and thus $G=G^{-1}$.  Hence $E(G)$ which is backward invariant under $G$ must also be forward invariant under $G$.  (4) is shown in~\cite{StankewitzRepelDense2}.

\end{proof}

Our focus will be on M\"obius semigroups and will require use of the following fundamental classification of maps in $\mathcal{M}$, the group of all M\"obius maps on $\CC$.
We recall the following definitions and results which can be found in
\cite{BeardonGroups}.
For every complex matrix
$$M=\left(\begin{array}{c c} a \ \  b\\c\  \ d \end{array}\right)$$
with $ad-bc\ne 0$ there is a corresponding $m(z)=(az+b)/(cz+d)$ in
$\mathcal{M}$.  We write $tr(M)=a+d$ and $det(M)=ad-bc$
for the trace and determinant of $M$.
Since $m(z)$ determines $M$ up to a scalar factor, we may
unambiguously define $tr^2[m]=tr^2(M)/det(M).$
This invariant of $m(z)$ determines its dynamical properties,
as follows.

\begin{theorem}[Classification of M\"obius maps $m \in \mathcal{M}$]
Let $Id$ denote the identity map on $\CC$ and
let $m \in \mathcal{M}$, $m \neq Id.$  Then
\begin{enumerate}
\item $tr^2[m]=4$ if and only if $m$ is \underline{parabolic}, i.e., $m$ has only one
  neutral fixed point and $m$ is conjugate to the translation $z\mapsto
  z+1$.
\item $tr^2[m]\in [0,4)$ if and only if $m$ is \underline{elliptic}, i.e., $m$ has
  two neutral fixed points and $m$ is conjugate to a rotation $z\mapsto
  kz$ with $|k|=1$.
\item $tr^2[m]\in (4,+\infty)$ if and only if $m$ is \underline{hyperbolic}, i.e., $m$ has
  an attracting and a repelling fixed point and $m$ is conjugate to
  $z\mapsto kz$ with $k\in \R$ and $|k|>1$.
\item $tr^2[m]\notin
  [0,+\infty)$ if and only if $m$ is \underline{strictly loxodromic}, i.e.,
$m$ has an attracting and a repelling fixed point and $m$ is conjugate to
$z\mapsto kz$ with $|k|>1$ and $k \notin \R$.
\end{enumerate}
\end{theorem}

Finally, a \textit{loxodromic} map is any $m \in \mathcal{M}$ which is either
hyperbolic or strictly loxodromic, i.e., a map having an attracting and a repelling fixed point.
Note, however, that this terminology is not universal.

\section{Full backward iteration algorithm}\label{SectionFull}

Let $a \in \CC$ be fixed.  For each $i_j \in \{1, \dots, k\}$, we write  $z_{i_1, i_2, \dots, i_n}=f^{-1}_{i_n}\circ \dots \circ f^{-1}_{i_1}(a)$.
Note that $z_{i_1, i_2, \dots, i_n}$ depends on the initial choice of $a$ though we suppress this dependence in our notation.

Throughout we assume $b=(b_1, \dots, b_k)$ is a probability vector, i.e., each $b_j>0$ and $\sum_{j=1}^k b_j =1$.

Denoting by $\delta_z$ the unit point mass measure at $z$, we  define probability measures $\mu_n^{a,b}$ on $\CC$ as follows:

$$\mu_1^{a,b}
=\sum_{i_1=1}^k b_{i_1} \delta_{f^{-1}_{i_1}(a)}$$
and, in general, for $n>1$
\begin{align*}
\mu_n^{a,b}
&= \sum_{i_1, i_2, \dots, i_n =1}^k b_{i_1}\cdots b_{i_n} \delta_{z_{i_1, i_2, \dots, i_n}}\\
&=\sum_{i_1, i_2, \dots, i_n =1}^k b_{i_1} \cdots b_{i_n} \delta_{f^{-1}_{i_n}\circ \dots \circ f^{-1}_{i_1}(a)}.
\end{align*}

Following~\cite{Boyd2} and~\cite{Su3}, we define a bounded linear operator $T=T_G^b$ on the space $C(\CC)$ of continuous functions (endowed with the sup norm $\|\cdot\|_\infty$) on $\CC$ by

\begin{equation}\label{Tdef}
(T \phi)(z) = \int_\CC \phi(s) \, d\mu_1^{z,b}(s) = \sum_{i_1=1}^k b_{i_1} \phi(f^{-1}_{i_1} z),
\end{equation}

noting $\|T\|=1$.  Hence, $(T \phi)(z)$ is a weighted average of $\phi$ evaluated at all $k$ preimages of $z$ under all generators $f_j$.

Letting $\mathcal{P}(\CC)$ denote the space of probability Borel measures on $\CC$, and noting that it is a compact metric space in the topology of weak* convergence,
we have that the adjoint $T_b^*:\mathcal{P}(\CC) \to \mathcal{P}(\CC)$  of $T^{b}$ is given by
\begin{equation}
(T_b^* \rho)(A)=\int \mu_1^{z,b}(A)\, d\rho(z)
\end{equation}
for all Borel sets $A \subseteq \CC$.  Using the operator notation $\langle \phi, \rho \rangle = \int \phi \, d\rho$ we express the action of the adjoint as $\langle T \phi, \rho \rangle = \langle \phi, T^* \rho \rangle$.  Note that the map $\CC \to \mathcal{P}(\CC)$ given by $z \mapsto \mu_1^{z,b}$ is continuous since for a sequence $z_n \to z_0$ in $\CC$, we have $\langle \phi, \mu_1^{z_n,b} \rangle = (T\phi)(z_n) \to (T\phi)(z_0)=\langle \phi, \mu_1^{z_0,b} \rangle$ for any $\phi \in C(\CC)$, i.e., $\mu_1^{z_n,b} \to \mu_1^{z,b}$.

\vskip.1in

We give a claim which is obtained by using results from \cite{SumiRandom}.
\begin{claim}\label{SumiClaim}
Let $G=\langle f_1, \dots, f_k \rangle$ be a M\"obius semigroup and let $b=(b_1, \dots, b_k)$ be a probability vector. When $J_{ker }(G)=\emptyset $ and $J(G)\neq \emptyset $, we have the following (a)(b)(c)(d)(e).

(a) $\sharp J(G)\geq 3.$

(b) There exists a unique minimal set $L$ of $G$, where we say that a non-empty compact subset $L$ of $\hat{\Bbb{C}}$ is a minimal set of $G$ if $L=\overline{\cup _{g\in G}\{ g(z)\} }$ for each $z\in L.$

(c) $L\cap F(G)\neq \emptyset .$

(d) $L=\overline{\{ z\in L\mid \exists g\in G \mbox{ s.t. } z \mbox{ is an attracting fixed point of } g\}}.$

(e) Let $b$ be a probability vector.
Then there exists a unique Borel probability measure $\nu ^{b}$ on $\hat{\Bbb{C}}$ such that $M^{n}(\phi )(z)\rightarrow \int \phi \, d\nu^{b}$ uniformly on $\hat{\Bbb{C}}$, where, $M=M_G^b$ is the transition operator with respect to $b$ given by $M(\phi)(z)= \sum_{j=1}^k b_j \phi(f_j(z))$, for any $\phi \in C(\CC)$.
  Also, $\nu^{b}$ is the unique Borel probability measure on $\CC$ such that $M^{\ast }\nu^{b}=\nu^{b}.$ Also, the support of $\nu^{b}$ is equal to $L.$
\end{claim}

\begin{proof}
We now apply parts (2), (3), (8), (10), (13), (17), and (21) of Theorem 3.15 in~\cite{SumiRandom} using the probability measure $\tau =\sum_{j=1}^k b_j \delta_{f_j}$ on the space of non-constant rational functions.  Part (a) follows from (3).  From (a) and Proposition~\ref{facts}(4) there exists a loxodromic map in $G$ and hence (b) follows from (21).  If (c) did not hold, then $L \subseteq J(G)$ would be contained in $J_{ker}(G)$ contradicting our assumption.  Part (d) follows from (17) noting that the $S_\tau$ in the reference is $L$ in this case.  Lastly, (e) follows from (2), (8), (10), (13), and (21).
\end{proof}

We note that the complexities involved in the proof of Theorem 3.15 of~\cite{SumiRandom} required a very delicate analysis based on the hyperbolic metric in this general setting (see Remark~\ref{Delicate}).
Now, we prove the following.

\begin{lemma}\label{Lset}
Let $G$ be a finitely generated Mobius semigroup such that $J_{ker}(G)=\emptyset$ and $J(G)\neq \emptyset.$ Then, the unique minimal set $L$ of $G$ is equal to $J(G^{-1}).$
\end{lemma}

\begin{proof}
By (d) in Claim~\ref{SumiClaim}, we have $L\subseteq J(G^{-1})$ since any attracting fixed point of any $g \in G$ must be a repelling fixed point of $g^{-1}$.

We now consider the following two cases.
Case (i) $\sharp L\geq 3$. Case (ii) $\sharp L\leq 2$.

Suppose we have case (i).
Then $G(L)\subseteq L$ (thus $(G^{-1})^{-1}(L)\subseteq L$), $\sharp L\geq 3.$ Since $J(G^{-1})$ is a minimal backward invariant compact subset under $G^{-1}$ which has at least three elements, we obtain that $L\supseteq J(G^{-1}).$ Hence $L=J(G^{-1}).$

Suppose we have case (ii).
By (c) in Claim~\ref{SumiClaim} we choose $z \in L \cap F(G)$.  Since $\# L < +\infty$ and $F(G)$ is forward invariant under $G$, we see that $L = \overline{G(z)} = G(z) \subset F(G).$ By (a) in Claim~\ref{SumiClaim}, we can take hyperbolic distance on each connected component of $F(G)$ and for a small $\epsilon >0$ let $V_{\epsilon}$ be the $\epsilon$-hyperbolic neighborhood of $L$
(we consider connected components $\{ W_{i}\}$ of $F(G)$ which meet $L$ and we consider the $\epsilon $-hyperbolic neighborhood $A_{i,\epsilon }$ of $W_{i}\cap L$ in $W_{i}$ and let
$V_{\epsilon }=\cup _{i}A_{i,\epsilon }$.) Then $(G^{-1})^{-1}(\overline{V_\epsilon}) =G(\overline{V_{\epsilon}})\subseteq \overline{V_{\epsilon}}$ by Pick's Lemma and hence $J(G^{-1})\subseteq \overline{V_{\epsilon }}.$ Since $\epsilon$ is an arbitrary small number, it follows that $J(G^{-1})\subseteq L.$ Hence $J(G^{-1})=L.$
\end{proof}

Note that replacing the maps $f_j$ by their inverses $f_j^{-1}$ in the definition of the operator $M$ of Claim~\ref{SumiClaim} produces the operator $T$ given in~\eqref{Tdef}.  Hence for a M\"obius semigroup $G$ we see that $M_{G^{-1}}^b=T_G^b$, where $G^{-1} = \langle f_1^{-1}, \dots, f_k^{-1} \rangle$ is the inverse semigroup of $G$.  Thus applying Claim~\ref{SumiClaim} and Lemma ~\ref{Lset} to the inverse semigroup $G^{-1}$ we have the following result which parallels results for non-M\"obius rational semigroups found in~\cite{Boyd2, Su3, BarnsleyDemko}.

\begin{theorem}\label{FullMethod}
Let $G=\langle f_1, \dots, f_k \rangle$ be a M\"obius semigroup and let $b=(b_1, \dots, b_k)$ be a probability vector.  Suppose $J_{ker}(G^{-1}) = \emptyset$ and
$J(G^{-1})\neq \emptyset .$
 Then the measures $\mu_n^{a,b}$ converge weakly to a  Borel probability measure $\mu^{b}=\mu^{b}_G$ on $\CC$ independently of and uniform in $a \in \CC$.  Further, the support of $\mu^{b}$ is $J(G)$ and $T_b^*\mu^{b} = \mu^{b}$.  Further, $\mu^{b}$ is the unique Borel probability measure on $\CC$ such that $T_{b}^{\ast }\mu ^{b}=\mu ^{b}$.
\end{theorem}

\begin{remark}
Theorem~\ref{FullMethod} provides the justification for the ``full backward iteration algorithm" used to graphically approximate $J(G)$.  This method simply plots the $k^n$ (not necessarily distinct) points in the support of $\mu_n^{a,b}$.  We note that this iterative process plots all $k$ inverses of each of the $k^{n-1}$ points in the support of $\mu_{n-1}^{a,b}$ to generate the support of $\mu_n^{a,b}$.  Also, note that the support of $\mu_n^{a,b}$ is independent of $b$ (but not of $a$), which merely adjusts the weights on the point masses.
\end{remark}

\section{Random backward iteration algorithm}

Let $\Sigma_k^+ = \prod_{n=1}^\infty \{1, \dots, k\}$ denote the space of one-sided sequences on $k$ symbols, regarded with the usual topology and $\sigma$-algebra of Borel sets.   Now let
$P_b$ be the Bernoulli measure on $\Sigma_k^+$ which is then given on basis elements as follows:  for fixed $j_n$ for $n=1, \dots, m$, we have $P_b(\{(i_1, i_2, \dots) \in \Sigma_k^+:i_n=j_n \textrm{ for all } n=1, \dots, m\})= b_{j_1} \cdots b_{j_m}$.

Let us define a random walk as follows.  Starting at a point $z_0=a \in \CC$, we note that there are $k$ preimages of $a$ under the generators of the semigroup $G$.  We randomly select $z_1$ to be $f_j^{-1}(a)$ with probability $b_j$.  Likewise, $z_2$ is randomly selected to be one of the $k$ preimages of $z_1$.  Continue in this fashion to generate what we call a \textit{random backward orbit} $\{z_n\}$ of $z_0=a$ under the semigroup $G$.

Utilizing the notation introduced in Section~\ref{SectionFull}, we see that $\Sigma_k^+$ generates the entire set of backward orbits starting at $z_0=a$ by letting the sequence $(i_1, i_2, \dots)$ generate the backward orbit $\{g_{i_n}\circ \dots \circ g_{i_1}(z_0)\}=\{z_{i_1, i_2, \dots, i_n}\}$.

Formally, we define our random walk $\{z_n\}$ in terms of random variables $\{Z_n\}$ given as follows.  For each $n \in \N$, we let $Z_n=Z_n^b:(\Sigma_k^+,P_b) \to \CC$ by $Z_n(i_1, i_2, \dots) = z_{i_1, i_2, \dots, i_n}$ and $Z_0 \equiv a$.  Hence, $Z_{n+1}(i_1, i_2, \dots)=f^{-1}_{i_{n+1}}Z_n(i_1, i_2, \dots)$.

The proof given for the analogous result in~\cite{RandomBackStankewitzSumi} shows the following.

\begin{claim}\label{Markov}
Let $G=\langle f_1, \dots, f_k \rangle$ be a M\"obius semigroup, let $a \in \CC$, and let $b=(b_1, \dots, b_k)$ be a probability vector.  Then the stochastic process $\{Z_n:n=0, 1, 2, \dots\}$ forms a Markov process with transition probabilities $\{\mu_1^{z,b}\}$, i.e., for each Borel set $B \subseteq \CC$ we have $P_b(\{Z_{n+1} \in B|Z_0, \dots, Z_n\}) = \mu_1^{Z_n,b}(B)$.
\end{claim}

The main result can now be stated in terms of the probability measures, defined for each $n \in \N$ and $(i_1, i_2, \dots) \in \Sigma_k^+$, by
\begin{equation}
\mu_{i_1, \dots, i_{n}}^{a} = \frac1n \sum_{j=1}^n  \delta_{z_{i_1, i_2, \dots, i_j}}.
\end{equation}

\begin{theorem}\label{main}
Let $G=\langle f_1, \dots, f_k \rangle$ be a M\"obius semigroup, let $a \in \CC$, and let $b=(b_1, \dots, b_k)$ be a probability vector.  Suppose $J_{ker}(G^{-1}) = \emptyset$ and
$J(G^{-1})\neq \emptyset .$
Then, for $P_b$ a.a. $(i_1, i_2, \dots) \in \Sigma_k^+$, the probability measures $\mu_{i_1, \dots, i_{n}}^{a}$ converge weakly to $\mu^{b}$ in $\mathcal{P}(\CC)$.
\end{theorem}

By using Theorem~\ref{FullMethod}, the proof of Theorem~\ref{main} is now nearly identical with that given for the proof of the analogous result in~\cite{RandomBackStankewitzSumi}, but is given in Appendix~\ref{AppProof} of this paper for completeness.  See~\cite{HawkinsTaylor} for the corresponding statement and proof in the case of iteration of one function of a rational map of degree two or more, which is the model for the analogous work here.

\begin{remark}
Theorem~\ref{main} provides the justification for the ``random backward iteration algorithm" used to graphically approximate $J(G)$.  This method simply plots, for large $n$, the $n$ points in the support of $\mu_{i_1, \dots, i_{n}}^{a}$, i.e., the points of a random backward orbit, where $(i_1, i_2, \dots) \in \Sigma_k^+$ is randomly selected according to $P_b$.  If $a \notin J(G)$, then it is often appropriate to not plot the first hundred or so points in the random backward orbit since the earlier points in the orbit might not be very close to $J(G)$.
\end{remark}
By Claim~\ref{SumiClaim}, Lemma~\ref{Lset} and \cite[Theorem 3.15]{SumiRandom}, we can show the following result.
\begin{theorem}
\label{converge}
Under the assumptions of Theorem~\ref{main}, let $a\in \CC $ and let $b=(b_{1},\ldots, b_{k})$ be a
probability vector. Then for $P_{b}$ a.a. $(i_{1},i_{2},\ldots )\in \Sigma _{k}^{+}$,
we have $d(z_{i_{1},\ldots i_{n}}, J(G))\rightarrow 0$ as $n\rightarrow \infty .$
\end{theorem}
\begin{proof}
By Claim~\ref{SumiClaim}, there exists a unique minimal set $L$ of $G^{-1}.$
By Lemma~\ref{Lset}, $L=J(G).$
By \cite[Theorem 3.15(7)]{SumiRandom}, the statement of our theorem holds.
\end{proof}

\begin{remark}
In Theorems~\ref{FullMethod} and \ref{main},
both conditions $J_{\ker}(G^{-1})=\emptyset $
and $J(G^{-1})\neq \emptyset$ are necessary.
We see this in the following.
\begin{itemize}
\item[(1)]
To see that the condition $J(G^{-1})\neq \emptyset $ is
necessary, let $k=1, f_{1}(z)=e^{2\pi i\theta }z$ where $\theta \in \Bbb{Q}$
is a constant, and $G=\langle f_{1}\rangle .$
Then $J(G^{-1})=\emptyset $ and $\mu ^{a}_{1,1,\ldots, 1}$
converges to a measure $\rho _{a}$ for each $a\in \CC $,
and $\rho _{a}\neq \rho _{a'}$ if $|a|\neq |a'|.$ Thus the condition
$J(G^{-1})\neq \emptyset $ is necessary.
\item[(2)]
To see that the condition $J_{\ker }(G^{-1})=\emptyset $
is necessary, let $k=2, f_{1}(z)=2z, f_{2}(z)=\frac{1}{2}z$,
and $G=\langle f_{1},f_{2}\rangle .$
Then $\emptyset \neq \{ 0,\infty \} \subset J_{\ker }(G^{-1})$,
and $\mu _{n}^{0,b}\rightarrow \delta _{0},
\mu _{n}^{\infty ,b}\rightarrow \delta _{\infty },
\mu ^{0}_{i_{1},\ldots, i_{n}}
\rightarrow \delta _{0}, $ and
$\mu ^{\infty }_{i_{1},\ldots ,i_{n}}\rightarrow \delta _{\infty }$
as $n\rightarrow \infty .$ Thus the condition $J_{\ker }(G^{-1})
=\emptyset $ is necessary.
\end{itemize}
Also, let $k\in \Bbb{N}$ be an arbitrary natural number and
suppose a M\"obius semigroup $G=\langle f_{1},\ldots, f_{k}\rangle
$ satisfies that
$\emptyset \neq E(G)\cap J(G)\subsetneqq J(G).$
Then we cannot have the conditions
$J_{\ker }(G^{-1})=\emptyset $ and
$J(G^{-1})\neq \emptyset $ simultaneously.
For, if we have both of them, then
Lemma~\ref{Lset} implies that
the unique minimal set $L$ of $G^{-1}$  is equal to
$J(G)$. Since $G^{-1}(E(G)\cap J(G))\subset E(G)\cap J(G)$,
we have a minimal set $L_{0}$ of
$G^{-1}$
in $E(G)\cap J(G).$ Thus $L_{0}=L=J(G).$ However,
this contradicts the assumption
$E(G)\cap J(G)\subsetneqq J(G).$
Thus we cannot have the conditions
$J_{\ker }(G^{-1})=\emptyset $ and
$J(G^{-1})\neq \emptyset $ simultaneously.
Also, the conclusion supp$\,\mu ^{b}=J(G)$ in Theorem~\ref{FullMethod}
cannot hold for any $a\in E(G)\cap J(G).$
To construct an example of M\"obius semigroup $G$
such that $\emptyset \neq E(G)\cap J(G)\subsetneqq J(G)$,
let $k=2, f_{1}(z)=2z$ and let $f_{2}$ be a M\"obius transformation
such that $0$ is an attracting fixed point of $f_{2}$ and $1$ is a
repelling fixed point of $f_{2}.$ Let $G=\langle f_{1},f_{2}\rangle .$
Then $\emptyset \neq \{ 0\} =E(G)\cap J(G)\subsetneqq J(G).$
\end{remark}
\begin{remark}
There are many examples of finitely generated
M\"obius semigroups $G$ such that
$J_{\ker }(G^{-1})=\emptyset $ and $J(G^{-1})\neq \emptyset .$
In fact, these conditions are checkable and it is easy for us to construct examples
of such M\"obius semigroups.  The class of such M\"obius semigroups is very
big. For various sufficient conditions for M\"obius semigroups $G$
to satisfy   $J_{\ker }(G^{-1})=\emptyset $ and $J(G^{-1})\neq \emptyset $
and examples of such M\"obius semigroups $G$, see
Section~\ref{Applications}.
\end{remark}

\section{Consequences of the main result Theorem~\ref{main}}

Using Theorem~\ref{main} together with the fact that the support of $\mu^b$ is $J(G)$ no matter what probability vector $b$ is chosen, one can quickly show the following results.
Under the assumptions of Theorem~\ref{main},
for each $a\in \CC $ and for each probability vector $b$, let
$\Sigma ^{a,b}:=\{ (i_{1},i_{2},\ldots )\in \Sigma _{k}^{+}\mid \mu _{i_{1},\ldots, i_{n}}^{a}\rightarrow \mu ^{b}
\mbox{ as }n\rightarrow \infty \}.$  Note that by Theorem~\ref{main}, we have $P_{b}(\Sigma ^{a,b})=1.$
Nearly identical details of the argument can be found in~\cite{HawkinsTaylor} and so we omit them here.

\begin{corollary}\label{Closure}
Under the assumptions of Theorem~\ref{main},
let $a \in \CC $ and $(i_1, i_2, \dots) \in \cup_b \Sigma^{a,b}$, where the union is taken over all probability vectors $b$.  Then
$$J(G) \subseteq \overline{\bigcup_{j=0}^\infty \{z_{i_1, i_2, \dots, i_j}\}}.$$
Furthermore, if $a \in J(G)$, then
$$J(G) = \overline{\bigcup_{j=0}^\infty \{z_{i_1, i_2, \dots, i_j}\}}.$$
\end{corollary}

\begin{remark}
Meeting the conclusion of this corollary
and Theorem~\ref{converge} is generally what one is looking for when saying that the ``random backward iteration" method works in drawing (an approximation of) $J(G)$.  Note that this conclusion holds for $P_b$ a.a. $(i_1, i_2, \dots) \in \Sigma_d^{+}$, regardless of the choice of $b$.
\end{remark}

\section{Applications}
\label{Applications}

In this section, we give some examples to which we can apply
Theorems~\ref{FullMethod}, \ref{main}, \ref{converge} and Corollary~\ref{Closure}.
For  M\"obius semigroups $S$, we give some sufficient conditions for $J_{ker }(S)$ to be empty.
\begin{lemma}\label{Lemma1}
Let $S$ be a M\"obius semigroup and $S^{-1}$ its inverse semigroup.  Assume $J(S^{-1}) \setminus J(S) \neq \emptyset$ and $E(S^{-1}) \cap J(S) = \emptyset$.  Then $J_{ker} (S) = \emptyset$.
\end{lemma}

\begin{proof}
Choose $z_0 \in J(S^{-1}) \setminus J(S)$.

Then for all $z \in \CC \setminus E(S^{-1})$ we see $\overline{S(z)} = \overline{(S^{-1})^{-1}(z)}$ must contain $z_0$ by Proposition~\ref{facts}\eqref{BackOrbitNonExcept}.  Therefore for all $z \in \CC \setminus E(S^{-1})$, we have $\overline{S(z)}$ meets $F(S)$ since $z_0 \in \CC \setminus J(S)= F(S)$.

For $z \in E(S^{-1})$, since $S(E(S^{-1})) = E(S^{-1})$ by Proposition~\ref{facts}\eqref{ExcSetCompInv} we see that the assumption $E(S^{-1}) \cap J(S) = \emptyset$ gives that $S(z) \subseteq E(S^{-1}) \subset F(S)$.

So, for all $z \in \CC$ we have $\overline{S(z)}$ meets $F(S)$.  Thus $J_{ker} (S) = \emptyset$ since $J_{ker} (S)$ is a forward invariant subset of $J(S)$.
\end{proof}

\begin{lemma}
Let $S$ be a M\"obius semigroup and suppose that $J(S^{-1})\neq \CC $, $E(S^{-1})\cap J(S)=\emptyset$,
and int$(J(S))\neq \emptyset $, where int denotes the set of interior points with respect to the topology in
$\CC .$ Then $J_{\ker }(S)=\emptyset .$
\end{lemma}
\begin{proof}
Suppose $J(S^{-1})\supset J(S).$
Since $S^{-1}(J(S))\subset J(S)$, we have
$S^{-1}(\mbox{int}(J(S)))\subset \mbox{int}(J(S)).$ Since $J(S^{-1})\neq \CC $,
Montel's theorem implies that $\mbox{int}(J(S))\subset F(S^{-1}).$ However,
this is a contradiction. Thus we have $J(S^{-1})\setminus J(S)\neq \emptyset .$
By Lemma~\ref{Lemma1}, we obtain that $J_{\ker }(S)=\emptyset .$
\end{proof}

\begin{lemma}
Let $S$ be a M\"obius semigroup and $S^{-1}$ its inverse semigroup.  Assume $\#J(S) \ge 3$ and $J(S^{-1}) \cap J(S) = \emptyset$.  Then $J_{ker} (S) = \emptyset$.
\end{lemma}

\begin{proof}
Suppose $w \in J_{ker} (S)$.  Using Proposition~\ref{facts}\eqref{RepelDense} we may choose $g \in S$ such that $g$ has a repelling fixed point $p \neq w$.  So $g$ has an attracting fixed point $q$ which must then be in $J(S^{-1})$.  Hence $q \notin J(S)$.  Since $g^n(w) \to q$ we contradict the fact that $J_{ker} (S)$ is both a closed forward invariant set under $S$ and contained in $J(S)$.
\end{proof}

We apply the above results to the Caruso semigroups $S_\beta$ defined as follows.
For a fixed nonzero complex number $\beta$ we let $f(z)=\beta + 1/z$ and
$g(z)=-\beta + 1/z.$ We define the M\"obius semigroup $S_\beta=\langle f,g \rangle$, and its inverse semigroup
$S_\beta'=\langle f^{-1}, g^{-1} \rangle.$  See~\cite{FMS} for background, but note that there the notation is different in that M\"obius \textit{groups} are denoted by $\langle \dots \rangle$ and \textit{semigroups} are denoted $\prec \dots \succ$.  We note that in Section~7 of~\cite{FMS} it is shown that neither $S_\beta$ nor $S_\beta'$ have finite orbits in $H^3 \cup \CC$ and so, in particular, $E(S_\beta') = \emptyset$ and $E(S_\beta) = \emptyset$.  Further it is shown that $\#J(S_\beta) \ge 3$ and $\#J(S_\beta') \ge 3$.

\begin{theorem}\label{Caruso}
For each $\beta \in \C \setminus \{0\}$ we have either

(i) $J(S_\beta) \neq J(S_\beta')$ and $J_{ker}(S_\beta) = \emptyset = J_{ker}(S_\beta')$

or

(ii) $J(S_\beta) = J(S_\beta')$ and $J_{ker}(S_\beta) = J(S_\beta) = J(S_\beta') = J_{ker}(S_\beta').$

\end{theorem}

\begin{proof}
From Proposition 7.14 of~\cite{FMS}, $J(S_\beta) \neq J(S_\beta')$ implies both $J(S_\beta) \setminus J(S_\beta') \neq \emptyset$ and $J(S_\beta') \setminus J(S_\beta) \neq \emptyset$.  Hence by Lemma~\ref{Lemma1} the result follows (noting that $J(S_\beta) = J(S_\beta')$ implies $J(S_\beta)$ is both forward and backward invariant under $S_\beta$ and thus $J_{ker}(S_\beta) = J(S_\beta) = J(S_\beta') = J_{ker}(S_\beta')$).
\end{proof}

Hence whenever $J(S_\beta) \neq J(S_\beta')$ we see by Theorem~\ref{main} (noting $\#J(S_\beta) \ge 3$ and $\#J(S_\beta') \ge 3$) that both the random backward iteration algorithm and the full backward iteration algorithm work for both $S_\beta$ and $S_\beta'$, and these method can start with any seed value $a \in \CC$.  This is an extension from what was shown in~\cite{FMS}, and which we now describe.

Let $K \subset \CC$ be a compact set that is forward invariant under each map in rational semigroup $S$.
Further suppose
$K$ is equipped with a metric $d,$ consistent with the topology on $K$. If there
exists a constant $c,$ $0<c<1,$ such that for each $s \in S$ we have
$d(s(z),s(w))\leq cd(z,w)$ for all $z, w \in K,$ then we say that $S$ is a
CIFS (contracting iterated function system) on $(K,d)$.  We define the \textit{attractor} $\mathcal{A}=\mathcal{A}(K,S)$
as the closure of the set of all fixed points in $K$ of the maps $s \in S$.  If $\textrm{Int}(K) \neq \emptyset$ and $\mathcal{A} \subset \textrm{Int}(K)$, then $\mathcal{A}$ is called
\textit{thick}.

Theorem 5.7 in~\cite{FMS} shows that for any finitely generated M\"obius semigroup $S$ the following (i)-(iii) are equivalent: (i) $S$ has a thick attractor $J(S^{-1})$, (ii) $S^{-1}$ has a thick attractor $J(S)$, and (iii) $J(S^{-1}) \cap J(S)=\emptyset$ and the generators of $S$ are loxodromic (having a repelling fixed point).  Noted in~\cite{FMS}, when a thick attractor exists both the (forward) random walk method and the (forward) full method work for drawing the corresponding Julia sets, however, the initial seed value must be chosen from the set $K$ given in the definition of thick attractor.

\begin{remark}\label{Delicate}
Theorem~\ref{main} presented here is not so restrictive as Theorem 5.7 in~\cite{FMS} on the conditions that $J(S^{-1}) \cap J(S)=\emptyset$ nor on the choice of starting seed value.  This, however, required a more delicate analysis as conducted in the proof of Theorem 3.15 in~\cite{SumiRandom} which was crucial in the proof of Theorem~\ref{main}.
\end{remark}

Concrete examples where Theorem~\ref{main} applies, but the result of~\cite{FMS} does not, are for $\beta = \pm 1 \pm i$ where (see Theorem 8.2 in~\cite{FMS}), $J(S_\beta)$ and $J(S_\beta')$ are unequal Cantor sets with $J(S_\beta) \cap J(S_\beta')=\{ \pm 1, \pm i\}.$  Other examples of this type are for $\beta = \pm 2$ or $\beta = \pm 2i$, where (see Theorem 8.1 in~\cite{FMS}) $J(S_\beta)$ and $J(S_\beta')$ are unequal Cantor sets with $J(S_\beta) \cap J(S_\beta')=\{ \beta/2, -\beta/2\}.$

We give further examples of M\"obius semigroups $G$ with $J(G)\cap J(G^{-1})\neq \emptyset $
to which we can apply Theorems~\ref{FullMethod}, \ref{main}, \ref{converge} and Corollary~\ref{Closure}.
\begin{example}
Denoting the open unit disk by $\mathbb{D}$ with boundary $S=\partial \mathbb{D}$, let $f_1, \dots, f_k$ be M\"obius maps with $f_j(\mathbb{D})=\mathbb{D}$ for each $j= 1, \dots, k$ chosen such that for $G=\langle f_1, \dots, f_k \rangle$ we have $J(G) \cap J(G^{-1}) \neq \emptyset$.  (For example, any parabolic fixed points of any $g \in G$ would necessarily lie in $J(G) \cap J(G^{-1})$.
Also, if $G=\langle f_{1},\ldots, f_{k}\rangle $ is a Fuchsian group with $J(G)\neq \emptyset$,
then $J(G)\cap J(G^{-1})=J(G)\neq \emptyset .$)  Note that $J(G)\subseteq S$ and $J(G^{-1})\subseteq S$.  Set $f_{k+1}(z)=az$ with $|a|>1$ and call $\tilde{G}=\langle f_1, \dots, f_k, f_{k+1} \rangle$.  Since $\CC \setminus \overline{\mathbb{D}}$ is forward invariant under $\tilde{G}$ we must have $J(\tilde{G}) \subseteq \overline{\mathbb{D}}$.  Similarly, $J(\tilde{G}^{-1}) \subseteq \CC \setminus \mathbb{D}$.  In this way since $J(G) \subset J(\tilde{G})$ and $J(G^{-1}) \subset J(\tilde{G}^{-1})$, we have that $J(\tilde{G}) \cap J(\tilde{G}^{-1})$ contains $J(G) \cap J(G^{-1})$, which could be constructed to be a rather large subset of $S$.  However,
Theorems~\ref{FullMethod}, \ref{main}, \ref{converge} and Corollary~\ref{Closure} still apply.
\end{example}

\vskip.3in
We conclude this paper with the following:

\begin{question}
Is $\{\beta \in \C \setminus \{0\}: J(S_\beta) \neq J(S_\beta')\}$ open and dense in $\C \setminus \{0\}$?
\end{question}

\vskip.3in
Acknowledgement.  This work was partially supported by a grant from the Simons Foundation (\#318239 to Rich Stankewitz). The research of the second author was partially supported by
JSPS KAKENHI 24540211, 15K04899.

\appendix

\section{Proof of Theorem~\ref{main}}\label{AppProof}

In this section we introduce a key result of Furstenberg and Kifer to relate the invariance of the measure $\mu^b$ to the convergence of Cesaro averages of $\phi(z_n)$ where $\{z_n\}$ is a random backward orbit and $\phi \in C(\CC)$.
As in~\cite{HawkinsTaylor, RandomBackStankewitzSumi}, this result is the key tool in the proof of our main theorem.

Let $Y$ be a compact metric space and let $\mathcal{P}(Y)$ be the space of Borel probability measures on $Y$, noting that $\mathcal{P}(Y)$ is a compact metric space in the topology of weak* convergence.  Suppose there exists a continuous map $Y \to \mathcal{P}(Y)$ assigning to each $x \in Y$ a measure $\mu_x$.  The corresponding Markov operator $H:C(Y) \to C(Y)$ is given by
$$Hf(x)= \int f(y) \, d\mu_x(y).$$

Suppose the stochastic process $\{X_n:n=0, 1, 2, \dots\}$ is a Markov process with state space $Y$ corresponding to $H$, i.e.,  $$P(\{X_{n+1} \in A|X_0, X_1, \dots, X_n\}) = \mu_{X_n}(A).$$
Given these assumptions we then have the following, which is a weaker version (but sufficient for our purposes) of Theorem 1.4 in~\cite{FurstKifer}.

\begin{theorem}[\cite{FurstKifer}]\label{FKtheorem}
Assume that there is a unique probability measure $\nu$ on $Y$ that is invariant under the adjoint operator $H^*$ on $\mathcal{P}(Y)$ and let $\phi \in C(Y)$.  Then with probability one
$$\frac1{N+1} \sum_{n=0}^N \phi(X_n) \to \int \phi \, d\nu$$
as $N \to \infty$.
\end{theorem}

\begin{proof}[Proof of Theorem~\ref{main}]
Let $\phi \in C(\CC)$.  We apply Theorem~\ref{FKtheorem} using $Y=\CC$, $\mu_x = \mu_1^{z,b}$, $H=T$, $\nu = \mu^b$ and $X_n = Z_n$, noting that all the hypotheses have been met by Theorem~\ref{FullMethod} and Claim~\ref{Markov}.  Thus we obtain a set $\Sigma_\phi \subseteq \Sigma_d^+$ with $P_b(\Sigma_\phi)=1$ such that for all $(i_1, i_2, \dots) \in \Sigma_\phi$ we have $$\langle \phi, \mu_{i_1, \dots, i_{n}}^a\rangle = \frac1{n} \sum_{j=1}^n \phi(z_{i_1, i_2, \dots, i_j}) \to \int \phi \, d\mu^b= \langle \phi, \mu^b \rangle.$$

Since the set $\Sigma_\phi$ depends on $\phi$, we require an extra step to achieve single such set to work for all maps in $C(\CC)$.

Since $\CC$ is compact we know that $C(\CC)$ is separable.  Let $\{\phi_j\}$ be
dense in $C(\CC)$.  Let $\Sigma_0=\cap_{j=1}^{\infty} \Sigma_{\phi_j}$ and
note that $P_b(\Sigma_0)=1$.
Let $\psi \in C(\CC )$.  Let $\epsilon>0$.
Select $\phi_j$ such that $\|\psi-\phi_j\|_{\infty} < \epsilon$.  Then for all $(i_1, i_2, \dots) \in \Sigma_0$, we have
$|{1 \over n}
\sum_{j=1}^n \psi(z_{i_1, i_2, \dots, i_j})
-\int_\CC \psi\,d\mu^b|
\leq |{1 \over n} \sum_{j=1}^n \psi(z_{i_1, i_2, \dots, i_j})
-{1 \over n} \sum_{j=1}^n \phi_j(z_{i_1, i_2, \dots, i_j})|\\
+|{1 \over n} \sum_{j=1}^n \phi_j(z_{i_1, i_2, \dots, i_j})
- \int_\CC \phi_j\,d\mu^b|
+ |\int_\CC \phi_j\,d\mu^b - \int_\CC \psi\,d\mu^b|\\
< \epsilon + |{1 \over n} \sum_{j=1}^n \phi_j(z_{i_1, i_2, \dots, i_j})
- \int_\CC \phi_j\,d\mu^b| +\epsilon
<3\epsilon$
for large $n$.

Hence on $\Sigma_0$ we get the convergence we seek,
for all $\psi \in C(\CC)$ and this completes the proof.
\end{proof}

%

\begin{thebibliography}{10}

\bibitem{BarnsleyFractalsEvery}
Michael Barnsley.
\newblock {\em Fractals everywhere}.
\newblock Academic Press, Inc., Boston, MA, 1988.

\bibitem{BarnsleyDemko}
Michael~F. Barnsley and Stephen~G. Demko.
\newblock Rational approximation of fractals.
\newblock In {\em Rational approximation and interpolation ({T}ampa, {F}la.,
  1983)}, volume 1105 of {\em Lecture Notes in Math.}, pages 73--88. Springer,
  Berlin, 1984.

\bibitem{BEH}
Michael~F. Barnsley, John~H. Elton, and Douglas~P. Hardin.
\newblock Recurrent iterated function systems.
\newblock {\em Constr. Approx.}, 5(1):3--31, 1989.

\bibitem{BeardonGroups}
Alan~F. Beardon.
\newblock {\em The geometry of discrete groups}, volume~91 of {\em Graduate
  Texts in Mathematics}.
\newblock Springer-Verlag, New York, 1995.
\newblock Corrected reprint of the 1983 original.

\bibitem{Boyd2}
David Boyd.
\newblock An invariant measure for finitely generated rational semigroups.
\newblock {\em Complex Variables Theory Appl.}, 39(3):229--254, 1999.

\bibitem{Julia2.0}
Trey Butz, Wendy Conatser, Ben Dean, Kristopher Hart, Yun Li, and Rich
  Stankewitz.
\newblock Julia 2.0 fractal drawing program.
\newblock https://github.com/bsumath/julia/wiki.

\bibitem{Elton}
John~H. Elton.
\newblock An ergodic theorem for iterated maps.
\newblock {\em Ergodic Theory Dynam. Systems}, 7(4):481--488, 1987.

\bibitem{FMS}
David Fried, Sebastian~M. Marotta, and Rich Stankewitz.
\newblock Complex dynamics of {M}\"obius semigroups.
\newblock {\em Ergodic Theory Dynam. Systems}, 32(6):1889--1929, 2012.

\bibitem{FurstKifer}
H.~Furstenberg and Y.~Kifer.
\newblock Random matrix products and measures on projective spaces.
\newblock {\em Israel J. Math.}, 46(1-2):12--32, 1983.

\bibitem{HawkinsTaylor}
Jane Hawkins and Michael Taylor.
\newblock Maximal entropy measure for rational maps and a random iteration
  algorithme.
\newblock {\em Internat. J. Bifur. Chaos Appl. Sci. Engrg.}, 13(6):1442--1447,
  2003.

\bibitem{HM1}
A.~Hinkkanen and G.J. Martin.
\newblock The dynamics of semigroups of rational functions {I}.
\newblock {\em Proc. London Math. Soc.}, 3:358--384, 1996.

\bibitem{HMJ}
A.~Hinkkanen and G.J. Martin.
\newblock {J}ulia sets of rational semigroups.
\newblock {\em Math. Z.}, 222(2):161--169, 1996.

\bibitem{Hutchinson}
John~E. Hutchinson.
\newblock Fractals and self-similarity.
\newblock {\em Indiana Univ. Math. J.}, 30(5):713--747, 1981.

\bibitem{FLM}
A.~Freire{,}~A. Lopes and R.~Ma{\~{n}}\'e.
\newblock An invariant measure for rational maps.
\newblock {\em Bol. Soc. Bras. Math.}, 14(1):45--62, 1983.

\bibitem{Lyu}
M.~Lyubich.
\newblock Entropy properties of rational endomorphisms of the {R}iemann sphere.
\newblock {\em Ergod. Th. \& Dynam. Sys.}, 3:351--385, 1983.

\bibitem{Mane}
R.~Ma{\~{n}}\'e.
\newblock On the uniqueness of the maximizing measure for rational maps.
\newblock {\em Bol. Soc. Bras. Math.}, 14(1):27--43, 1983.

\bibitem{Mihailescu}
Eugen Mihailescu.
\newblock Approximations for {G}ibbs states of arbitrary {H}\"older potentials
  on hyperbolic folded sets.
\newblock {\em Discrete Contin. Dyn. Syst.}, 32(3):961--975, 2012.

\bibitem{RSThesis}
Rich Stankewitz.
\newblock {\em Completely invariant {J}ulia sets of rational semigroups}.
\newblock Ph.D. Thesis. University of Illinois, 1998.

\bibitem{StankewitzRepelDense2}
Rich Stankewitz.
\newblock Density of repelling fixed points in the {J}ulia set of a rational or
  entire semigroup, {II}.
\newblock {\em Discrete Contin. Dyn. Syst.}, 32(7):2583--2589, 2012.

\bibitem{RandomBackStankewitzSumi}
Rich Stankewitz and Hiroki Sumi.
\newblock Random backward iteration algorithm for {J}ulia sets of rational
  semigroups.
\newblock {\em Discrete Contin. Dyn. Syst.}, 35(5):2165--2175, 2015.

\bibitem{Su3}
Hiroki Sumi.
\newblock Skew product maps related to finitely generated rational semigroups.
\newblock {\em Nonlinearity}, 13:995--1019, 2000.

\bibitem{SumiRandom}
Hiroki Sumi.
\newblock Random complex dynamics and semigroups of holomorphic maps.
\newblock {\em Proc. Lond. Math. Soc. (3)}, 102(1):50--112, 2011.

\end{thebibliography}

\end{document}